\documentclass[12pt]{amsart}
\usepackage{amsmath,amsfonts,amscd,bezier,amsthm,amssymb}
\usepackage{mathrsfs,mathtools}
\usepackage{color}
\usepackage{cite}
\usepackage{graphicx}
\usepackage{float}
\usepackage{enumitem,xfrac,nccmath}
\usepackage{tasks, multicol,yhmath}

\usepackage[pagewise]{lineno}%\linenumbers

\newtheorem{theorem}{Theorem}[section]
\newtheorem{lemma}[theorem]{Lemma}
\newtheorem{corollary}[theorem]{Corollary}
\theoremstyle{definition}

\newtheorem{remark}[theorem]{Remark}

\newtheorem{conjecture}[theorem]{Conjecture}

\newcommand{\R}{\mathbb{{R}}}
\newcommand{\Z}{\mathbb{{Z}}}
\newcommand{\N}{\mathbb{{N}}}
\newcommand{\T}{\mathbb{{T}}}
\newcommand{\C}{\mathbb{{C}}}
\newcommand{\D}{\mathbb{{D}}}

\renewcommand{\Re}{\mathfrak{Re}}
\renewcommand{\Im}{\mathfrak{Im}}

\DeclareMathOperator{\essen}{ess}

\numberwithin{equation}{section}

\allowdisplaybreaks

\makeatletter
\DeclareRobustCommand{\eqrefp}[2]{%
	\textup{\tagform@{\refp{#1}{#2}}}}
\makeatother
\DeclareRobustCommand{\refp}[2]{%
	\expandafter\ifx\csname r@#1\endcsname\relax
	\textbf{??}%
	\else
	\edef\areferencia{\ref{#1}-)}%
	\expandafter\eqrefpaux\areferencia-#2%
	\fi}
\def\eqrefpaux#1-#2){#1}

\begin{document}
	\title[Descriptions of cyclic functions]{Towards spectral descriptions of cyclic functions}

        \author[M. Monsalve-López]{Miguel Monsalve-López}
	\address[M. Monsalve-López]{Departamento de Análisis Matemático y Matem\'atica Aplicada, Facultad de Ciencias Matemáticas, Universidad Complutense de Madrid, Spain}
	\email{migmonsa@ucm.es}

	\author[D. Seco]{Daniel Seco}
	\address[D. Seco]{Departamento de An\'alisis Matem\'atico e IMAULL,  Universidad de La Laguna, Avenida Astrof\'isico Francisco S\'anchez, s/n, 38206 San Crist\'obal de La Laguna, Santa Cruz de Tenerife,  Spain} \email{dsecofor@ull.edu.es}

	\thanks{The first named author has been supported by Plan Nacional I+D Grant no. PID2022-137294NB-I00 (Spain). The second named author is funded through grant RYC2021-034744-I by the Ram\'on y Cajal programme from Agencia Estatal de Investigaci\'on (Spanish Ministry of Science, Innovation and Universities).}

	

	\date{\today}
	
	\begin{abstract} 
	
We build on a characterization of inner functions $f$ due to Le, in terms of the spectral properties of the operator $V=M_f^*M_f$ and study to what extent the cyclicity on weighted Hardy spaces $H^2_\omega$ of the function $z \mapsto a-z$ can be inferred from the spectral properties of analogous operators $V_a$. We describe several properties of the spectra that hold in a large class of spaces and then, we focus on the particular case of Bergman-type spaces, for which we describe completely the spectrum of such operators and find all eigenfunctions.
	\end{abstract}
	
	\maketitle
	
\section{Introduction}

In classical complex function theory, the concepts of inner and cyclic functions play a prevalent role. A function $f\colon \D \rightarrow \C$ is called \emph{inner} if its modulus is bounded by 1 at every point and its radial limit boundary values have modulus equal to 1 at almost every point of the boundary. A theorem of Beurling provides a complete description of invariant subspaces under the shift operator in the Hardy space $H^2$ based on inner functions and this is a cornerstone of modern complex analysis. A function is called \emph{cyclic} (in a space $H$ of analytic functions) if it is contained in no proper invariant subspace of $H$. Thanks to a factorization obtained by Smirnov, Beurling's Theorem also implies a complete description of cyclic functions in $H^2$. For more on this topic, we refer the reader to \cite{Gar}.
In trying to develop a similar theory suited to spaces of analytic functions other than $H^2$, several generalizations of inner functions have been proposed since the work of Rudin in the 1960s. One of these concepts is particularly based on the inner product properties that make inner functions special for invariant subspaces of a given Hilbert space $H$:
We say $f \in H$ is a \emph{$H-$inner} function if it has norm 1 in $H$ and \[\langle f , z^j f \rangle = 0, \qquad j \neq 0.\]
It is easy to check that, under mild assumptions, the normalized orthogonal projection of the constant function $1$ onto an invariant subspace is a $H$-inner function and this has numerous consequences \cite{Remarks}. Recall that a  \emph{multiplier} $f$ in a space $H$ is a function for which the operator $M_f$ of multiplication by $f$ acts boundedly on $H$. We are going to study properties of $M_f$ through the mirror of its precomposition with its adjoint $M^*_f$, acting on certain function spaces that are described as follows.
Throughout this article $\omega=\{\omega_n\}_{n \in \N}$ will denote a monotonic positive sequence satisfying $\omega_0=1$, $\lim_{n\rightarrow \infty} \frac{\omega_{n+1}}{\omega_n} = 1$, and \begin{equation}\label{eqnA1}
\sum_{n=0}^\infty \left(1-\frac{\omega_{n+1}}{\omega_{n}}\right)^2  < \infty.\end{equation} We will then say that $\omega$ is a \emph{valid} weight. In that case, we will denote by \[C_0= \sup_{n \in \N} \frac{\omega_{n+1}}{\omega_n}; \quad C_1= \sup_{n \in \N} \frac{\omega_{n}}{\omega_{n+1}}.\]
Finally, denote by $H^2_\omega$ the space of holomorphic functions $f$ over the unit disc of the complex plane with Maclaurin coefficients $\{a_n\}_{n\in \N}$ satisfying 
\begin{equation}\label{eqnB1}
\|f\|^2_\omega := \sum_{n=0}^\infty |a_n|^2 \omega_n < \infty.
\end{equation} $H^2_\omega$ is called a \emph{weighted Hardy space} with weight $\omega$. The conditions on the space guarantee several properties: \begin{enumerate}
\item[(P1)] Polynomials form a dense subset. \item[(P2)] Both the \emph{shift} operator $M_z$ and its left-inverse are bounded in $H^2_\omega$ (with norms given in terms of $C_0$ and $C_1$). \item[(P3)] The unit disc is the common domain for the elements of the space. \item[(P4)] The assumption \eqref{eqnA1} guarantees that we can make use of an important theorem in spectral theory while still including a large class of relevant examples. \item[(P5)] By the monotonicity, either $C_0=1\leq C_1$ or $C_1=1\leq C_0$. \item[(P6)] The space $H^2_\omega$ is an example of a reproducing kernel Hilbert space over the disc (see next Section for the definition). \end{enumerate} Bergman-type spaces $A^2_\beta$ ($\beta > -1$) will attract some of the focus in the present article. $H^2_\omega$ where $\omega_n = \binom{n+ \beta +1}{n}^{-1}$.  With this weight, the norm is also given by 
\[\|f\|^2_\omega= (\beta+1) \int_{\D} |f(z)|^2 (1-|z|^2)^\beta dA(z),\] where $dA$ denotes the normalized element of area measure. We prefer the notation $\alpha = -1-\beta$, which allows to extend the range of spaces further to include not only the values for $\alpha <0$ but up to $\alpha <1$. For $\alpha \geq 0$, these are usually referred to as Dirichlet-type spaces, but our choice of norm is only equivalent to the usual norm in such spaces, with the exception of the special values $\alpha = -1$ (referred to as the Bergman space, $A^2$), and $\alpha=0$ (the Hardy space, $H^2$). These weights are non-decreasing when $\alpha \in [0,1)$ and non-increasing when $\alpha \leq 0$. The limit case $\alpha=1$ is not included in this notation but it can be understood to correspond to the well studied classical Dirichlet space, where $\omega_k = k+1$, but after studying the previous cases in detail we will make comments about this particular space. In previous work by our second author \cite{Seco}, a characterization of $H$-inner functions for the Dirichlet space in terms of multiplication properties was found, building on previous work by Richter and Sundberg \cite{RS93}. This characterization was then expressed in operator-theoretic terms by Le \cite[Theorem 2.8]{Le} in a strikingly simple form that is valid in a huge class of reproducing kernel Hilbert spaces:
\begin{theorem}[Le] \label{TLe}
Let $\omega$ be valid and $f$ be a multiplier on $H^2_\omega$. Then the following are equivalent:
\begin{enumerate}
\item[(a)] $f$ is $H^2_\omega$-inner.
\item[(b)] $M_f^*M_f 1 =1$.
\end{enumerate}
\end{theorem}

On the other hand, cyclic functions play a dual role to that of $H$-inner functions, since they are exactly those for which the orthogonal projection  of the function 1 onto their corresponding invariant subspace is, again, the function 1. A complete description of cyclic functions has not even been achieved in long studied spaces such as the \emph{Dirichlet space} $D$, in which the Brown-Shields conjecture (a proposed characterization) has been open for over 40 years, becoming a fundamental problem in the fields of Complex Analysis and Operator Theory. A good reference for this problem is \cite{EFKMR}. Cyclic functions in weighted Hardy spaces were studied in \cite{FMS1}. The theorems there and in Le's work actually do not need (or mention) the assumption \eqref{eqnA1}.

However, it is well understood when a \emph{polynomial} is cyclic in a space. A polynomial is cyclic in $H^2_\omega$ if and only if it has no zeros at the points where point evaluation is a bounded functional (essentially, either $\D$, for spaces with $\sum_{j=0}^\infty \frac{1}{\omega_j} = + \infty$; or $\overline{\D}$, for spaces where that sum is convergent).

Since $H$-inner functions $f$ can be described, as in Theorem \ref{TLe}, in terms of the eigenvalue properties of the operator $M_f^* M_f$, it seems reasonable to ask the following: 
\[\textbf{Main question:} \quad \textit{What information can be obtained from $M_f^* M_f$ on whether $f$ is cyclic?}\] Unless $f$ is constant, if part (b) of the above Theorem is satisfied then $f$ is not cyclic, but it seems plausible that weaker necessary conditions may be found.

With the long term goal of identifying such features, a toy problem we will focus on here is that of \emph{describing the spectrum, point spectrum, and corresponding eigenfunctions, of $M_f^* M_f$} for a family of functions of the simplest form possible containing inner and cyclic functions: $f(z) = a-z$ for $a \in \mathbb{C}$. We would expect that the spectral information would contain information on whether $f$ is cyclic and in particular have a transition in terms of the values of $a$ according to whether $a \in \mathbb{D}$, $a \in \T$ or $a \in \mathbb{C} \backslash \overline{\mathbb{D}}$. This will be apparent in our Theorem \ref{superthm}, where (in any $H^2_\omega$ space and for any $a \in \C \backslash \{0\}$) we describe the spectrum completely except for the determination of two positive sequences that may compose the point spectrum. The essential spectrum is then the interval $[(1-|a|)^2, (1+|a|)^2]$, and indeed, we find out that $a\in \T$ happens if and only if $V_a:=M^*_{a-z}M_{a-z}$ is not invertible, that is, $0 \in \sigma(V_a)$.
Then we will obtain descriptions of the point spectra as well as all eigenfunctions in Theorems \ref{thm:dalphanegativo}, \ref{thm:dalpha01},  below, covering the values $\alpha < 0$ and $\alpha \in (0,1)$, respectively. Since cyclicity does not depend on the choice of equivalent norm on a particular Hilbert space, we feel entitled to study the phenomenon with our preferred choice of equivalent norms, that will have the advantage of making eigenvalue equations into rather simple ODEs. This simplicity breaks when we study the Dirichlet space, and hence, we will see some limitations to the scope of our results.  In a previous work, the integrability of differential equations arising from some extremal problems relevant to cyclic functions was made easier under the Bergman-type choice of equivalent norm in \cite{RevMatIber}. We observe some similarity between the two processes and we believe there is a deeper connection between the two phenomena. 

Unfortunately, the two ranges of spaces we will study according to the sign of $\alpha$ will send contradicting messages: for $\alpha <0$ the transition that will happen at the unit circle will make the point spectrum change from a convergent sequence to the empty set; meanwhile, for $\alpha \in (0,1)$, the point spectrum when $a \in \C \backslash \D$ will be formed by the union of two convergent sequences. Therefore, there seems to be no easy necessary or sufficient condition for cyclicity to be identified in terms of the point spectra.
When it comes to studying the Hardy space $H^2$ ($\omega \equiv 1$), it should be noted that the lack of point spectrum is well established, and that a theorem of Hartman and Wintner \cite{HartWint} gives every answer needed: We denote by $T_\varphi$ the Toeplitz operator with symbol $\varphi \in L^\infty(\T)$. This is the operator on $H^2$ consisting of multiplying by $\varphi$ and then projecting from $L^2$ onto $H^2$. With this notation, and for any bounded analytic function $f$, our operator $M^*_f M_f$ is just $T_{|f|^2}$, and Hartman and Wintner described the spectra of all self-adjoint Toeplitz operators:
\begin{theorem}
Let $\varphi: \T \rightarrow \R$ with $\varphi \in L^\infty$. Then \[\sigma(T_\varphi) = \sigma_{ess}(T_\varphi) = [\essen \inf_{\theta \in [0,2\pi)} \varphi (e^{i \theta}) , \essen \sup_{\theta \in [0,2\pi)} \varphi (e^{i \theta})].\] 
\end{theorem}

Two remarks are in order at this point: firstly, in the case of the Hardy space, the operator $M^*_fM_f$ depends only on the modulus of $f$ on the boundary and so, it cannot detect whether $f$ has an inner part or not. In this sense, the answer to our main question for this space is, clearly, there is no relation whatsoever with cyclicity. However, this obstruction only happens exactly in $H^2$, and we will accordingly observe different spectral behaviour for the operator $M^*_f M_f$ when acting on $H^2$ compared to anywhere else.

So the reader may be asking \emph{why should anyone be optimist about ever finding any such information in other spaces?} Despite the $H^2$ negative result, the truncations of $M^*_f M_f$ to the finite dimensional spaces $\mathcal{P}_n$ encapsulate all information about the cyclicity of the function except for the value of $f(0)$. This can be seen from the description of optimal polynomial approximants (polynomials studied on \cite{FMS1} and which characterize the cyclic behaviour) since these are obtained from a linear system in which the only information is the action of $M^*_f M_f$ on the orthogonal basis of $H^2_\omega$ given by monomials and the value of $f(0)$. The matrices studied in that context do not correspond exactly with the truncations of the operator as we will understand it, due only to the lack of normalization of the basis of monomials.

On a final comment, it would be reasonable to normalize $f$ to have supremum modulus equal to 1 over the disc. This could lead to connections with de Branges-Rovnyak spaces or the general theory of contractions on Hilbert spaces, where the defect operator $(I-M^*_bM_b)^{1/2}$ is known to play an important role. In our results, this could be interpreted by dividing each eigenvalue by $(|a|+1)^2$ and performing a related transformation on the eigenfunctions. This could simplify the appearance of the spectrum. However, the natural norm for such normalization in spaces smaller than $H^2$ seems to be the one associated with multipliers and since this norm is more complicated to work with in general, we decided to stay with the form of $f$ we chose initially. Some of our results can be interpreted as giving efficient lower bounds for the multiplier norms of $a-z$ in various spaces, but in Bergman-type spaces these norms are well known.

The plan of this article is as follows: we start in Section \ref{prelim}, by presenting basic properties of the spaces and the norms that we will use as well as generalities about reproducing kernel Hilbert spaces and multiplication operators acting there; then, in Section \ref{recusect}, for each space $H^2_\omega$ and value of $a \in \C \backslash \{0\}$, we show how to obtain a recurrence relation for the coefficients of an eigenfunction. We continue with an exploration of the whole spectrum in Section \ref{completespec}, achieving then a complete description of the spectral properties up to describing the point spectrum. When one particularizes in the case of our Bergman-type norms, the recurrence relations will lead to a homogeneous first order differential equation. This differential equation will be solved completely in Section \ref{pointspecsect} for each of the spaces. We explore the problem in the Dirichlet space in Section \ref{Dirichlet}, where the main issue will be a lack of homogeneity of the ODE obtained, leading us to the study of hypergeometric functions. We conclude in Section \ref{furthersect} with some further remarks and suggestions for future research.

\section{Preliminaries}\label{prelim}

Consider a weighted Hardy space $H^2_\omega$. The associated inner product is
\[
    \langle f,g\rangle_{H^2_\omega} \coloneqq \sum_{n = 0}^\infty a_n \overline{b_n} \omega_n
\]
for each $f(z) = \sum_{n = 0}^\infty a_n z^n$ and $g(z) = \sum_{n = 0}^\infty b_n z^n$ in the space. This is a clear example of \emph{reproducing kernel Hilbert spaces} over the unit disc $\D$. That means that the evaluation map $f \mapsto f(\lambda)$ is a bounded linear functional on $H_\omega^2$ for each $\lambda \in \D$. Accordingly, there exists an element of $H^2_\omega$, called the \emph{reproducing kernel} $\kappa_\lambda \in H^2_\omega$ such that \[\langle f,\kappa_\lambda\rangle = f(\lambda), \qquad \forall f \in H^2_\omega.\]  In fact, it is an easy exercise to check that $\kappa_\lambda$ is given by 
\[\kappa_\lambda (z) = \sum_{j=0}^\infty \frac{(\overline{\lambda}z)^j}{\omega_j}, \qquad \forall z \in \D.\] 
Several classical and well-known examples of weighted Hardy spaces are the Hardy space $H^2$ (with weight $\omega_n \equiv 1$), with norm also given by
\[
	\|f\|_{H^2} = \sup_{0 < r < 1}\bigg(\int_{0}^{2\pi} |f(re^{i\theta})|^2\, \mathrm{d}\theta\bigg)^{1/2};
\]
the Bergman space $A^2$ (which is the weighted Hardy space with weight $\omega_n = \tfrac{1}{n+1}$), whose norm is also given by
\[
	\|f\|_{A^2} = \bigg(\int_{\mathbb{D}}|f(z)|^2\, \mathrm{d}A(z)\bigg)^{1/2},
\]
where $dA$ denotes the normalized Lebesgue area measure on $\D$; and the Dirichlet space $D$ (with weight $\omega_n = n+1$), with norm
\[
	\|f\|_{D} \coloneqq \Big(\|f\|^2_{H^2} + \|f'\|^2_{A^2}\Big)^{1/2} < +\infty.
\]
Regardless of the choice of equivalent norm, a relevant and well-studied class of spaces is that of the \emph{Dirichlet-type spaces} $D_\alpha$ (for $\alpha \in \R$), corresponding to $H^2_\omega$ for the weight $\omega_n = (n+1)^\alpha$. The norm thus defined is often denoted $\|\cdot\|_\alpha$, and the cases $\alpha=-1, 0, 1$ correspond, respectively with $A^2$, $H^2$ and $D$. Furthermore, there is a continuous inclusion $D_\alpha \subsetneq D_\beta$ whenever $\beta<\alpha$ (indeed, note that $\|f\|_\beta \leq \|f\|_\alpha$ for every $f \in D_\alpha$) and for any $\alpha$ and $f \in Hol(\D)$ we have $f \in D_\alpha$ if and only if $f' \in D_{\alpha-2}$. For each $\alpha <1$, the weight $\omega_k = \binom{n-\alpha}{n}^{-1}$ gives an equivalent norm in the Dirichlet-type space $D_\alpha$ and we will refer to this norm 

\medskip

On the other hand, we will make use of the class of \emph{weighted Bergman spaces} $A^2_{\beta}$ (for $\beta > - 1$) of functions $f \in Hol(\D)$ such that
\[
	\|f\|_{A^2_\beta} = \bigg((\beta + 1)\int_{\mathbb{D}}|f(z)|^2 (1-|z|^2)^\beta\, \mathrm{d}A(z)\bigg)^{1/2}.
\]

Notice that the $\|\cdot\|_{A^2_\beta}$ norm is equivalent to $\|\cdot\|_\alpha$ whenever $\alpha=-1-\beta <0$ and it is well known that $A^2_\beta$ is a weighted Hardy space with weight $\omega_n = \binom{n-\alpha}{n}^{-1}$. In fact, this weight can be used to define a Dirichlet-type space equivalent norm for $\alpha <1$. For $\alpha >-1$, one can also choose the equivalent norm given by weight $\binom{n+\alpha}{n}$. These are the two different choices our approach will be specifically good for, although we will focus on the former and make use of the latter only when $\alpha = 1$ down below.

A holomorphic function $\varphi$ is called a \emph{multiplier} on a space $H$ if for each $f \in H$, the pointwise product $\varphi f$ belongs to $H$. Each multiplier $\varphi$ defines a linear bounded multiplication operator $M_\varphi \colon H \to H$, whose norm
\[
	\|M_\varphi\|_{H \to H} = \|\varphi\|_{\mathrm{Mult}(H)}
\]
is the \emph{multiplier norm} of $\varphi$ on $H$. The space of all multipliers with this norm, $\mathcal{M}_H$, is a Banach space. It is obviously closed under multiplication and if $1 \in H$ (like in our case, for all $H^2_\omega$ spaces), then $\mathcal{M}_H \subset H$. Even though the norm on $\mathcal{M}_H$ depends on the choice of equivalent norm imposed on $H$, the actual elements do not depend on this choice. Hence, when $H=D_\alpha$, we denote $\mathcal{M}_H \coloneqq \mathcal{M}_\alpha$. The class of polynomials $\mathcal{P}$ is contained in $\mathcal{M}_\alpha$, for all $\alpha \in \R$. For more general functions, there is a clear transition as $\alpha$ changes in $\R$ in terms of multipliers: for $\alpha \leq 0$, $\mathcal{M}_\alpha = H^{\infty}$ (the Banach space of bounded analytic functions with the uniform norm); as $\alpha$ changes from $0$ to $1$ it becomes increasingly difficult for a function to be a multiplier, but for $\alpha >1$, elements of the space $D_\alpha$ are so regular that they are automatically multipliers. A complete characterization for $\alpha \in (0,1]$ was achieved by Stegenga \cite{Stegenga} but this is not elementary and  we skip the details. A sufficient condition for a function $f \in D$ to be a multiplier of $D$ is that $f' \in H^\infty$, and this is enough for the purposes of this paper, since we will only deal with multiplication by polynomials.

\section{Recurrence relations}\label{recusect}

We denote the canonical basis $e_n(z) = z^n/\sqrt{\omega_n}$. This is an orthonormal basis for $H^2_\omega$ (and in fact, formed by $H^2_\omega$-inner functions). Recalling the above notation, from now on, fixed $a \in \C$, we denote by $V_a$ the operator 
\[V_a \coloneqq M^*_{a-z}M_{a-z}.\] We will use the convention that $h_n$ denotes the $n$-th Maclaurin coefficient of a holomorphic function $h$ and $h_{-1}\coloneqq 0$. Eigenfunctions for $V_a$ have several special properties, and we start with the following:

\begin{lemma}\label{lemma-recurrencia}
    Let $a \in \mathbb{C} \setminus\{0\}$ and consider the operator $V_a$ acting on $H_\omega^2$. For each $\lambda \in \sigma_{\text{\normalfont p}}(V_a)$, the associated eigenfunction $h(z)= \sum_{n \geq 0} h_n z^n$ comes given by the unique (up to normalization) solution to the recurrence relation
    \[
    	|a|^2h_n  -  ah_{n+1}\frac{\omega_{n+1}}{\omega_n}  -  \overline{a}h_{n-1}  +  h_n\frac{\omega_{n+1}}{\omega_n} = \lambda h_n, \quad \text{for each } n \geq 0.
    \]
\end{lemma}

It seems adequate to point out that when $a=0$, the Lemma still holds except for the fact that the equation cannot be considered a recurrence relation.

\begin{proof}
    Let $\lambda \in \mathbb{C}$ and consider $h(z) = \sum_{n = 0}^\infty h_n z^n \in H^2_\omega$. Since $e_n(z) = z^n/\sqrt{\omega_n}$ form an orthonormal basis on $H^2_\omega$, one has
    \[
    	\begin{aligned}
    	V_a h(z)&  = \sum_{n = 0}^\infty \langle M_{a-z}^*M_{a-z} h, e_n\rangle e_n(z) \\
    	& = \sum_{n = 0}^\infty \langle ah(z)-zh(z), az^n-z^{n+1}\rangle \frac{z^n}{\omega_n}. 	\end{aligned}
    \]
Now, simply distributing the inner products, and using the convention that $h_{-1} = 0$, gives
\[
    	V_a h(z) = \sum_{n = 0}^\infty \big(|a|^2h_n\omega_n  -  ah_{n+1}\omega_{n+1}  -  \overline{a}h_{n-1}\omega_n  +  h_n\omega_{n+1} \big)\frac{z^n}{\omega_n}.
    \]
     Hence, the eigenvalue equation \[V_a h(z) = \lambda h(z)\] becomes equivalent to the recurrence relation in the statement of the Lemma.
\end{proof}

\section{Essential and point spectrum}\label{completespec}

The reader may be wondering whether the argument of $a$ plays any role. It turns out that the problem of determining the point spectrum of $V_a$ is rotationally symmetric.

\begin{lemma}\label{thm:apositivo}
    Let $a \in \mathbb{C}\setminus\{0\}$ and consider $V_a$ acting on $H_\omega^2$. Then, $V_a \sim V_{ae^{i\theta}}$ for each $\theta \in \R$. In particular,
    \[
        \sigma(V_a) = \sigma(V_{ae^{i\theta}}) \quad \text{and} \quad \sigma_{\text{\normalfont p}}(V_a) =  \sigma_{\text{\normalfont p}}(V_{ae^{i\theta}}), \qquad \text{for all } \theta \in \mathbb{R}.
    \]
\end{lemma}

\begin{proof}
    Given $\theta \in \R$, consider the surjective linear operator $R_\theta \colon H_\omega^2 \to H_\omega^2$ given by $R_\theta f(z) = f(e^{i\theta}z)$. It is clear that $R_\theta$ is a unitary operator since it preserves the inner product:
    \[
        \langle R_{\theta} f, R_\theta g \rangle = \sum_{n = 0}^\infty f_ne^{in\theta} \overline{g_n e^{in\theta}} \omega_n = \sum_{n = 0}^\infty f_n\overline{g_n} \omega_n = \langle  f,  g \rangle.
    \]
    Let us check that $R_\theta$ is precisely the unitary transformation making the two operators similar. Consider any $h \in H^2_\omega$. Since $e_n(z) = z^n/\sqrt{\omega_n}$ form an orthonormal basis in $H_\omega^2$, we have
    \[
        \begin{aligned}
            R_\theta^*V_aR_\theta h(z) &  = \sum_{n = 0}^\infty \langle R_\theta^*M_{a-z}^*M_{a-z}R_\theta h, e_n \rangle e_n(z) \\
            & = \sum_{n = 0}^\infty \big\langle (a-z)h(e^{i\theta}z) ,(a-z)e^{in\theta}z^n\big\rangle \frac{z^n}{\omega_n}. 
        \end{aligned}
    \]

Just distributing the inner products, and using the same convention as before, yields 
   \[ \begin{aligned}
            R_\theta^*V_aR_\theta h(z) & = \sum_{n = 0}^\infty \Big(  e^{in\theta}h_n(|a|^2 \omega_n + \omega_{n+1}) - ae^{i(n+1)\theta}h_{n+1}\omega_{n+1} -\overline{a} e^{i(n-1)\theta}h_{n-1}\omega_n \Big)\frac{e^{-in\theta}z^n}{\omega_n} \\
		& = \sum_{n = 0}^\infty \Big( |a|^2 h_n\omega_n - ae^{i\theta}h_{n+1}\omega_{n+1} -\overline{a}e^{-i\theta} h_{n-1}\omega_n + h_n\omega_{n+1}\Big)\frac{z^n}{\omega_n}.
        \end{aligned}
    \]
    That amounts to saying
    \[
        R_\theta^*V_aR_\theta h(z) = V_{ae^{i\theta}}h(z).
    \]
    Therefore, $V_a = M_{a-z}^* M_{a-z}$ is unitarily equivalent to $ V_{ae^{i\theta}}$ and so they share the spectrum as well as all the spectral parts.
\end{proof}

The operator $V_a$ being self-adjoint guarantees that the spectrum will be real. In fact, we can at this point check that, in very large generality, $V_a$ has only positive point spectrum.

\begin{corollary}\label{coropositive}
Let $a \in \C \backslash \{0\}$, $\omega$ be a valid weight, and $V_a$ act on $H^2_\omega$. Then 
\[
    \sigma_{\text{\normalfont p}} (V_a) \subset (0, \infty).
\]
\end{corollary}

\begin{proof}
$V_a$ is self-adjoint and thus $\sigma_{\text{\normalfont p}}(V_a) \subset \R$. Let us check that any eigenvalue $\lambda$ must be strictly positive. By Lemma \ref{thm:apositivo}, we can assume $a>0$. We thus fix $\omega$, $a>0$ and $\lambda \leq 0$, and suppose $h(z)= \sum h_n z^n$ satisfies the recurrence relation from Lemma \ref{lemma-recurrencia} with $h_0=1$. By induction, this implies all $h_n$ are real. Define \[B_n := (ah_n - h_{n-1}) \omega_{n},\] (with $B_0=a >0$). The recurrence relation becomes then 
\[B_{n+1} = aB_n - \lambda h_n \omega_n.\]
Since $\lambda \leq 0$, this implies, firstly, that $h_n$ and $B_n$ are positive for all $n\in \N$: indeed, this holds for $n=0$ and if it does for $n$ then \begin{equation}\label{eqn21}B_{n+1} > aB_n.\end{equation}
This implies that \begin{equation}\label{eqn22}ah_{n+1} \omega_{n+1} > aB_n + h_n \omega_{n+1}.\end{equation}
From \eqref{eqn21}, if $a>1$ we obtain that $B_n$ grows exponentially, but this means that at least one of $|h_n|, |h_{n-1}|$ is comparably large and thus, $h$ cannot be holomorphic in $\D$. On the other hand, from \eqref{eqn22} we can see that $h_{n+1} > \frac{h_n}{a}$, which also means exponential growth if $a<1$, and lack of holomorphicity.
Finally, we can assume $a=1$, and \eqref{eqn21} implies that $B_n$ is non-decreasing, which means that \[h_n \geq \frac{1}{\omega_n} + h_{n-1} \geq \frac{1}{\omega_n} +1.\] 

Therefore, \[\|h\|^2_{H^2_\omega} \geq \sum_{n=0}^\infty \frac{1}{\omega_n} + \sum_{n=0}^\infty \omega_n.\] At least one of these two last sums must be infinite and therefore $h$ cannot be in $H^2_\omega$.
\end{proof}

We already noticed that the point spectrum depended on the choice of norm. However, the rest of the spectrum does not depend on the choice of equivalent norm or even (to some extent) on the choice of space. This will be a consequence of the Blumenthal-Weyl criterion and the Lieb-Thirring bound, as described in Theorem 1 of \cite{killip-simon-annals}:

\begin{theorem}\label{thmKilSim}
Let $J$ denote the infinite Jacobi matrix given by sequence $J_{n,n+1}=J_{n+1,n}=a_n$ and $J_{n,n}=b_n$ for $n \geq 1$. If \[2\sum_{n=1}^\infty (a_n-1)^2 + \sum_{n=1}^\infty b_n^2 < \infty,\]
then the essential spectrum of $J$ is $[-2,2]$. The point spectrum is then contained in two sequences, $\lambda^+$ and $\lambda^-$, converging to $\pm 2$ from outside of the interval, with both $\lambda_j^+ -2$ and $\lambda^-_j +2$ belonging to $\ell^{3/2}$.
\end{theorem}

We will make use of this in order to describe completely the spectrum of $V_a$, provided one can describe the pure point spectrum. This complements the computations to be done for some choices of weight $\omega_n = \binom{n-\alpha}{n}^{-1}$ when $\alpha <1$.

\begin{theorem}\label{superthm}
	Let $a \in \mathbb{C}\setminus\{0\}$ and $\omega$ be a valid sequence.
Consider the operator $V_{a} = M_{a-z}^*M_{a-z}$ acting on $H^2_\omega$. Then:
	\[
			\sigma(V_{a}) = \sigma_{\text{\normalfont{p}}}(V_{a}) \cup \big[(1-|a|)^2, (1+|a|)^2\big].
		\]

  Moreover, the point spectrum is contained in two positive sequences $\lambda^-$ converging to $(1-|a|)^2$ from below, and $\lambda^+$ converging to $(1+|a|)^2$ from above. Both sequences have differences with their limits in $\ell^{3/2}$.
\end{theorem}

\begin{proof} As usual, we assume $a > 0$. Notice $V_a$ is self-adjoint and thus, its discrete spectrum is just its pure point spectrum. The positivity of the point spectrum was shown in Corollary \ref{coropositive}. Let us find the essential spectrum. By Theorem \ref{thmKilSim}, we just need to check that $V_a$, seen as an infinite matrix representing its action on normalized monomials,  is a Hilbert-Schmidt perturbation of $(a^2+1)I - aJ_0$, where $J_0$ is the free Jacobi matrix, given by taking $a_n=1$ and $b_n=0$ for all $n \geq 1$.
In fact, fixed a weight $\omega$ and a value $a >0$, the infinite matrix $J$ associated to the operator $V_a$ is a Jacobi matrix given by
\[a_n= -a \sqrt{\frac{\omega_{n}}{\omega_{n-1}}},\]
\[b_n = a^2 + \frac{\omega_{n}}{\omega_{n-1}},\]
for all $n \geq 1$.
Thus for $J - (a^2+1)I+aJ_0$ to be a Hilbert-Schmidt perturbation of 0, we need to check that
\[2a^2 \sum \left(1-\sqrt{\frac{\omega_{n}}{\omega_{n-1}}}\right)^2+ \sum \left(1-\frac{\omega_{n}}{\omega_{n-1}}\right)^2 < \infty.\]
Since $\lim_{n \rightarrow \infty} \omega_n/\omega_{n-1} = 1$, using $1-\sqrt{x}= \frac{1-x}{1+\sqrt{x}}$, one just needs to check \eqref{eqnA1} is met.
\end{proof}

We can actually provide bounds on the potential eigenvalue sequences mentioned in the previous Theorem, and in particular show that if $\omega$ is non-increasing, $\lambda^+$ is empty.

\begin{lemma}\label{lowupbounds}
Let $a>0$, $h$ be an eigenfunction for some $\lambda$ for $V_a$ acting on $H^2_\omega$, with $\omega$ valid. Then:
\begin{enumerate}
\item[(I)] If $\omega$ is non-decreasing, then \[\lambda \in [(a-1)^2 -C_0a, (a+1)^2 + (C_0-1)a].\]
\item[(II)] If $\omega$ is non-increasing, then \[\lambda \in [(a-1)^2 -(C_1-1), (a-1)^2).\]
\end{enumerate}

\end{lemma}

\begin{proof}
We show the proof of the upper bounds for $\lambda$. The same principles as in case (I) give the lower bounds. From Lemma \ref{lemma-recurrencia} we have 
\[h_{n+1} = h_n \left(\frac{a^2+1-\lambda}{a} \frac{\omega_n}{\omega_{n+1}} + \frac{1}{a}(1 -\frac{\omega_n}{\omega_{n+1}})\right) - h_{n-1} \frac{\omega_n}{\omega_{n+1}}.  \]
We are going to show that $\{|h_n|\}$ is an exponentially growing sequence and thus $h$ cannot be in $Hol(\D)$. As a basis for induction, we can thus assume that $|h_n|>|h_{n-1}|$. In case (I), suppose \[\lambda = (a+1)^2 + (C_0-1+\varepsilon)a,\] with $\varepsilon >0$. Since $1 - \frac{\omega_n}{\omega_{n+1}} \geq 0$, we have
\[|h_{n+1}|> \left(|h_n| (C_0+1+\varepsilon) - |h_{n-1}|\right)\frac{\omega_n}{\omega_{n+1}}.\] By induction, this is larger than $|h_n|(1+\frac{\varepsilon}{C_0})$ and we are done.
In case (II), notice that \[\|(a-z)h\| \leq |a|\|h\| + \|zh\|,\] and the shift is a contraction. Thus $\|V_a\| \leq \|M_{a-z}\|^2 \leq (|a|+1)^2$. In the notation of Theorem \ref{superthm}, we then know that $\lambda$ must be part of $\lambda^-$. \end{proof}

For the particular case of $H^2$, the case (II) of Lemma \ref{lowupbounds} implies that the point spectrum is empty, recovering Hartman-Wintner's result. In the next section we solve the remaining questions on Bergman-type spaces.

\section{A complete description of point spectrum and eigenfunctions}\label{pointspecsect}

For the choice of norms we made for $A^2_\beta$ spaces, \eqref{eqnA1} holds because if $\omega_n = \binom{n-\alpha}{n}^{-1}$ then \[1-\frac{\omega_{n+1}}{\omega_{n}}= \frac{-\alpha}{n+1-\alpha}.\]

We recall some notations and introduce new ones that are stable during this Section. A number $a \in \C \setminus \{0\}$ is considered fixed, giving rise to a function $f$ with $f(z) = a-z$, for which we study the operator $V_a=M^*_f M_f$ acting on some space. Whenever we try to find whether certain complex number $\lambda$ belongs to the point spectrum of $V_a$ we will denote by $h$ a function that should satisfy the corresponding eigenvalue equation, with Taylor coefficients $\{h_n\}_{n \in \N}$ so that $h(z) = \sum_{n\in \N} h_n z^n$. In any expression regarding these coefficients, it is understood that $h_{-1}=0$. Once we fix a $\lambda \in \C$, we will denote by \[b= \frac{|a|^2+1-\lambda}{|a|}.\] In that case, $P$ will be the polynomial given by 
\[P(z) \coloneqq z^2-bz+1.\]
Obviously $P$ has 2 roots (counting multiplicity) that we denote by $c$ and $1/c$ and, without loss of generality, we assume that $|c| \geq 1$. A relevant property relating $b$ and $c$ is that
\begin{equation}\label{eqn1001}
c+1/c=b.\end{equation}
We will sometimes use it in the form
\begin{equation}\label{eqn1012}
c^2=bc-1.\end{equation}

We will start by recalling the disappointing conclusion with regards to our aim of identifying cyclic functions: there is nothing in the point spectrum in $H^2$. Now, the description of the point spectrum is not as trivial in other spaces. Let us deal with the classical Bergman-type spaces $A^2_\beta = D_\alpha$ for negative values of $\alpha=-1-\beta$. Recall that the weight in this context comes given by $\omega_n=\binom{n-\alpha}{n}^{-1}$.

\begin{theorem}\label{thm:dalphanegativo}
    Let $a \in \mathbb{C}\setminus\{0\}$, $\alpha < 0$ and $V_{a}$ acting on $A^2_{-1-\alpha}$.     \begin{enumerate}[label=\normalfont{(\roman*)}]
        \item If \(0 < |a| < 1\),
        \[
		\sigma_{\text{\normalfont p}}(V_{a}) = \Bigg\{\lambda_{j} \coloneqq |a|^2 + 1 + \frac{\alpha^2 - (2j +2 - \alpha)\varrho_{j}}{2(j+1)(j+1-\alpha)} \colon j \in \mathbb{N}\Bigg\},
	\]
	where \(\varrho_{j} \coloneqq \sqrt{\alpha^2 +4|a|^2(j+1)(j+1-\alpha)}\). Every eigenspace is one-dimensional and
    \[
        \ker(V_{a} - \lambda_{j}I) = \mathrm{span}\left\{\frac{\Big(z -\mfrac{\varrho_{j} + \alpha}{2\overline{a}(j+1)}\Big)^j}{\Big(1 - \mfrac{\varrho_{j} + \alpha}{2\overline{a}(j+1)}z\Big)^{j+2-\alpha}}\right\}.
    \]
        \item If \(|a| \geq 1\),  \(\sigma_{\text{\normalfont p}}(V_{a}) = \emptyset\).
    \end{enumerate}
\end{theorem}

Before we proceed with the proof of the Theorem, let us remark that this result points to the possible existence of a necessary condition for a (non-constant) function to be cyclic, that the point spectrum of the operator be empty.  However, we will later see that such condition is not necessary in other cases.

\begin{proof}
    Without loss of generality, using Lemma \ref{thm:apositivo}, we may assume that $a> 0$. Our plan is to first infer a differential equation that must be satisfied by any eigenfunction, then solve said differential equation to find all possible eigenfunctions $h$, and then determine which $h$ among those actually belong to $H^2_\omega:=A^2_{-1-\alpha}$ and are therefore actual eigenfunctions. As a result we will also obtain the corresponding eigenvalues. 
 
 Let us now start by choosing an arbitrary $\lambda \in \mathbb{C}$ and consider $h(z) = \sum_{n \geq 0} h_n z^n$ in $H^2_\omega$. Using that $\omega_n = \binom{n-\alpha}{n}^{-1}$ and operating in the recurrence relation obtained in Lemma \ref{lemma-recurrencia}, we arrive to
	\[
		h_{n+1} -  \frac{h_n}{a} = \frac{n+1 -\alpha}{n+1} \bigg(h_n\bigg(a -\frac{\lambda}{a}\bigg) - h_{n-1}\bigg).
	\]
We isolate the terms that are multiplied by $\alpha$, yielding
	\[
		(n+1)\cdot (h_{n+1} -  bh_n + h_{n-1}) = -\alpha \bigg(h_n\bigg(a -\frac{\lambda}{a}\bigg) - h_{n-1}\bigg).
	\]
Multiplying both sides by $z^n$ and summing over all $n \geq 0$, we may transform the recurrence relation into an equation on holomorphic functions:
	\begin{equation}\label{eqnD0}
		\sum_{n=0}^{\infty} z^n (n+1)\cdot (h_{n+1} -  bh_n + h_{n-1}) = -\alpha \sum_{n=0}^{\infty}z^n \bigg(h_n\bigg(a -\frac{\lambda}{a}\bigg) - h_{n-1}\bigg).
	\end{equation}
The right-hand side above is $\alpha (z+1/a -b) h(z)$. The left-hand side can be identified as $h' -b (zh)' + (z^2h)'$. Therefore, separating the terms on $h$ from those on $h'$, we have a first-order differential equation, 
	\[
		h'(z)P(z) = h(z)\bigg(b(1-\alpha) + \frac{\alpha}{a} + (\alpha - 2)z\bigg).
	\]
 
We are going to solve this equation by solving for $h'/h$ for which we express the resulting rational function in simple fractions: denote by $c$ and $1/c$ be the complex roots of $P$ and assume, for a moment, that $c \neq 1/c$ (i.e. that $c \neq \pm 1$), and $|c| \geq 1$. Thus, 
	\[
		\frac{h'(z)}{h(z)} = \frac{\mu}{z-c} + \frac{\nu}{z-\tfrac{1}{c}}
	\]
	where $\mu, \nu \in \mathbb{C}$ are given by the system
 \begin{equation}\label{eqn1002}
 \begin{pmatrix} 1  &  1 \\ -1/c & -c \end{pmatrix} \begin{pmatrix} \mu \\ \nu \end{pmatrix} = \begin{pmatrix} \alpha-2 \\ \frac{\alpha}{a} + b(1-\alpha) \end{pmatrix}.\end{equation}
This can be solved using basic linear algebra to yield
 \begin{equation}\label{AB-dalpha}
		\begin{cases}
			\mu & = -1 - \mfrac{\alpha}{a} \mfrac{a-c}{c^2 - 1} = -1 - \mfrac{\alpha}{a} \mfrac{a-c}{bc - 2},\\[1em]
			\nu & = \alpha - 1 +\mfrac{\alpha}{a} \mfrac{a-c}{c^2 - 1} = \alpha - 1 +\mfrac{\alpha}{a} \mfrac{a-c}{bc - 2}.
		\end{cases}
	\end{equation}
At this point, we can find a general solution to our differential equation, and this is
	\begin{equation}\label{h-dalpha}
		h(z) = k \left(1-\frac{z}{c}\right)^\mu\left(1- cz\right)^\nu, \quad \text{for } k \in \mathbb{C}.
	\end{equation}
	Consequently, for $\lambda \in \mathbb{C}$ to be an eigenvalue of $V_{a} \colon H^2_\omega \to H^2_\omega$ it is necessary and sufficient that the function $h(z)$ computed in \eqref{h-dalpha} belongs to $H^2_\omega$. Depending on the value of $c$ there are three cases:

\medskip
    \begin{itemize}
		\item\textbf{Case 1:} $|c| = 1$, $c \neq \pm 1$.
	\end{itemize}

\smallskip

A standard fact about reproducing kernel Hilbert spaces gives a maximum speed of growth towards the boundary for functions in the space: Applying Cauchy-Schwarz and the reproducing property, if $h \in H^2_\omega$ we must have \[|h(z)| \leq \|h\| \sqrt{\kappa_z(z)} \approx (1-|z|)^{\frac{\alpha-1}{2}}.\] This limit is surpassed by $h$ here. Indeed, by the first equation in the system \eqref{eqn1002} we cannot have both $\Re(\mu), \Re(\nu) > \tfrac{\alpha - 1}{2}$. That means that one of the singularities (the one at the point $c$ or that at $1/c$) must be too strong for $h$ to belong to $H^2_\omega$. Thus, we conclude that this case does not produce any eigenvalue.

\medskip
    \begin{itemize}
		\item\textbf{Case 2:} $|c|>1$.
	\end{itemize}

\smallskip

In this case, we must have $\nu \in \mathbb{N}$: otherwise, $h$ is not even holomorphic at the point $1/c \in \D$. Fix now $j \in \N$ and assume that $\nu = j$. Then \eqref{AB-dalpha} is rewritten as
	\[
		j +1 - \alpha = \frac{\alpha}{a} \frac{a-c}{bc - 2}.
	\]
This equation already contains much of the essential information about $\lambda$ but we need to disentangle $\lambda$ from the remaining terms. We want to solve for $\lambda$ in terms of $a$, $j$, and $\alpha$.
 
Hence, we continue by grouping the terms on $c$,
	\[
		c\bigg(\frac{j+1-\alpha}{\alpha}ab+1\bigg) = a\bigg(\frac{2(j+1)}{\alpha} - 1\bigg).
	\]
	Now, observe that \(\tfrac{j+1-\alpha}{\alpha}ab+1\neq 0\): otherwise, the right-hand side is also null and thus, $\tfrac{\alpha}{2} = j+1$ which cannot hold because $\alpha < 0$. Therefore, we can solve for $c$, yielding
	\begin{equation}\label{c-formula1-dalpha}
		c = a\frac{\tfrac{2(j+1)}{\alpha} - 1}{\tfrac{j+1-\alpha}{\alpha}ab+1}.
	\end{equation}
	Reminding that $c = (b\pm\sqrt{b^2 -4})/2$, one can equal this expression with \eqref{c-formula1-dalpha}, subtract $b/2$ and take squares on both sides to obtain 
	\[
		b^2- 4 = \left(2a\frac{\tfrac{2(j+1)}{\alpha} - 1}{\tfrac{j+1-\alpha}{\alpha}ab+1} -b\right)^2.
	\]
Opening the parentheses, cancelling the two adding terms on $b^2$, and multiplying both sides by the adequate quantity, this is equivalent to
	\[
		\bigg(\frac{j+1-\alpha}{\alpha}ab +1\bigg)^2 = ab\bigg(\frac{2(j+1)}{\alpha} - 1\bigg)\bigg(\frac{j+1-\alpha}{\alpha}ab +1\bigg) - a^2\bigg(\frac{2(j+1)}{\alpha} - 1\bigg)^2.
	\]
    Now, since we are interested in isolating $\lambda$ and $ab = a^2 + 1 -\lambda$, we rewrite our later expression gathering each summand in terms of powers of \(ab\):
    \[
    a^2b^2\frac{(\alpha -(j+1))(j+1)}{\alpha^2} - ab + 1 + a^2\bigg(\frac{2(j+1)}{\alpha} - 1\bigg)^2 = 0.
    \]
    Solving this quadratic equation and using $ab = a^2 + 1 -\lambda$, one has two possible solutions
    \[
        a^2 + 1 -\lambda = \alpha^2\cdot\frac{1 \pm \sqrt{1 - 4\left(1 + a^2\Big(\mfrac{2(j+1)}{\alpha} - 1\Big)^2\right)\mfrac{(\alpha -(j+1))(j+1)}{\alpha^2}}}{2(\alpha -(j+1))(j+1)}.
    \] We can then solve for $\lambda$. It may seem like a miracle but, after inserting the $\alpha^2$ factor inside the square root, we obtain a degree 4 polynomial on $\alpha$ that has a double root at $\alpha = 2j+2$. To avoid too much of a lengthy expression we denote \[\varrho_{j} = \sqrt{\alpha^2 +4a^2(j+1)(j+1-\alpha)}.\] Observe that \(\varrho_j > |\alpha|\). Doing so, one obtains the next two candidates for the eigenvalues
    \begin{equation}\label{eigenvaluescandidates-dalpha}
        \lambda= a^2 + 1 + \frac{\alpha^2 \pm (2j +2 - \alpha)\varrho_{j}}{2(j+1)(j+1-\alpha)}.
    \end{equation}

    Notice the negative value of the square root will give the actual eigenvalue stated in the Theorem for $\lambda_j$. On the other hand, we need to discard the value with the positive root, in order to exclude those false candidates which are not actual solutions of the eigenvalue equation. When we took squares of $\sqrt{b^2-4}$, at some earlier point, the roles of $c$ and $1/c$ were intertwined. We need to identify the case where $c$ has a correct value with $|c|>1$. Since \(ab = a^2 + 1 -\lambda\), one has
    \[
        ab =  \frac{-\alpha^2 \mp (2j+2-\alpha)\varrho_{j}}{2(j+1)(j+1 - \alpha)},
    \]
    which allows us to rewrite \(\tfrac{(j+1-\alpha)}{\alpha}ab + 1\) as
    \[
        \frac{(j+1-\alpha)}{\alpha}ab + 1 = \bigg(1\mp\frac{\varrho_{j}}{\alpha}\bigg)\bigg(1- \frac{\alpha}{2(j+1)}\bigg).
    \]
    Now, substituting this expression into \eqref{c-formula1-dalpha} gives
    \begin{equation}\label{eq:cambio-demo-dalpha}
        c = \frac{a\Big(\mfrac{2(j+1)}{\alpha} - 1\Big)}{\Big(1\mp\mfrac{\varrho_{j}}{\alpha}\Big)\Big(1- \mfrac{\alpha}{2(j+1)}\Big)}  = \frac{2a(j+1)}{\alpha \mp \varrho_{j}}.
    \end{equation}
    Hence, in order to have $|c|>1$ (and then, lie in \textbf{Case 2}), we need \(2a(j+1) > |\alpha \mp \varrho_{j}|\):

\noindent\(\bullet\) The choice of the positive root \(+\varrho_{j}\) for the eigenvalue in \eqref{eigenvaluescandidates-dalpha} gives
        \[
            c = \frac{2a(j+1)}{\alpha - \varrho_{j}}.
        \]
        Since \(|\alpha - \varrho_{j}| =  \varrho_{j} - \alpha\), we need to solve the inequality \(2a(j+1)>\varrho_{j} - \alpha\). This inequality can be simplified to \(4a\alpha(j+1) > -4a^2\alpha(j+1)\), which can never happen since  \(\alpha < 0\) and so this case must be discarded.

\noindent\(\bullet\) The choice of the negative root \(-\varrho_{j}\) for the eigenvalue in \eqref{eigenvaluescandidates-dalpha} gives
        \[
            c = \frac{2a(j+1)}{\alpha + \varrho_{j}}.
        \]
        Since \(|\alpha + \varrho_{j}| =  \varrho_{j} + \alpha\), we need to solve the inequality \(2a(j+1)>\varrho_{j} + \alpha\). This inequality is equivalent to \(-4a\alpha(j+1) > -4a^2\alpha(j+1)\) which holds if and only if \(0 < a < 1\) and so the eigenvalue description in the statement holds with the corresponding eigenfunctions. In this case, $h$ is a holomorphic function with a polynomial type root at $1/c$ and analytic up to the disc of radius $|c|$. Any function that is holomorphic in a disc of radius larger than 1 is automatically in any $H^2_\omega$ space and thus $h$ is an actual eigenfunction whenever $0 < a < 1$.

\newpage

    \begin{itemize}
		\item\textbf{Case 3:} $c = \pm 1$.
	\end{itemize}
 
 \smallskip
 
To complete the proof, let us finally consider the case that \(c = \pm 1\) (which implies $b = \pm 2$ and $\lambda = (a \mp 1)^2$). In this case, the differential equation becomes
    \[
        \frac{h'(z)}{h(z)} = \mp\frac{\alpha-2}{1 \mp z } + \frac{\alpha\big(\tfrac{1}{a} \mp 1
        \big)}{(1\mp z)^2},
    \]
    whose general solution is
    \begin{equation}\label{hz-c1}
        h(z) = k (1\mp z)^{\alpha - 2}\exp\bigg(-\frac{\alpha\big(\tfrac{1}{a} \mp 1\big)}{z\mp 1}\bigg), \quad \text{for } k \in \mathbb{C}.
    \end{equation}
This is a holomorphic function in $\D$. Thus, the points \(\lambda = (a \mp 1)^2\) will be eigenvalues of \(V_{a} \colon H^2_\omega \to H^2_\omega\) whenever the corresponding function \(h(z)\) belongs to \(H^2_\omega\). Bearing in mind that
    \[
        \begin{aligned}
            \|h\|^2_{\omega} & \sim \int_{\D}|h(z)|^2 (1-|z|)^{-\alpha - 1} \, \mathrm{d}A(z)\\& \sim \int_{\D} \frac{1}{|z \mp 1|^{-2\alpha + 4}} \exp\left(-2\alpha\bigg(\frac{1}{a} \mp 1\bigg)\Re\bigg(\frac{1}{z\mp 1}\bigg)\right)  (1-|z|^2)^{-\alpha - 1} \, \mathrm{d}A(z),
        \end{aligned}
    \]
    when \(a \geq 1\), the exponential within the integral is bounded from below. This tells us that $h \not\in H^2_\omega$, upon applying standard Forelli-Rudin estimates (see, for instance, \cite[Thm. 1.7]{HKZ-bergmanspaces}). Otherwise, when \(0 < a < 1\) we have \(-2\alpha\big(\tfrac{1}{a}\mp 1\big) > 0\). Thus, the choice \(c = -1\) must also be excluded since
    \[
        -2\alpha\bigg(\frac{1}{a} + 1\bigg) \Re\bigg(\frac{1}{z+1}\bigg) > -\alpha \bigg(\frac{1}{a} + 1\bigg),
    \]
    and so the latter argument can be applied again to see that $h \not\in D_\alpha$. On the other hand, for \(c = 1\) using, for instance the change of variables
    \[
        \begin{cases}
            \Re(z) & = \bigg(1 - \dfrac{1}{2\gamma}\bigg) + \dfrac{1}{2\gamma}\cos(\varphi),\\
            \Im(z) & = \dfrac{1}{2\gamma}\sin(\varphi),
        \end{cases}\qquad \text{with } \,\gamma > \frac{1}{2} \,\text{ and }\, 0<\varphi < 2\pi,
    \]
    one has that \(\mathrm{d}A(z) = \mfrac{1}{4\pi\gamma^3} (1-\cos(\varphi))\, \mathrm{d}\varphi \mathrm{d}\gamma\) and so, applying the identities
    \[
        \Re\bigg(\frac{1}{z-1}\bigg) = -\gamma, \qquad \frac{1 - |z|^2}{|1 - z|^2} = 2\gamma -1 \quad \text{and} \quad \frac{1}{|1 - z|^2} = \frac{2\gamma^2}{1 - \cos(\varphi)},
    \]
    this yields
    \[
        \|h\|^2_{\omega} \gtrsim \int_1^T \int_\theta^{2\pi -\theta} \gamma^3(2\gamma - 1)^{-\alpha - 1} e^{2\alpha \gamma\big(\tfrac{1}{a} - 1\big)} \frac{1}{(1 - \cos(\varphi))^2} \, \mathrm{d}\varphi \mathrm{d}\gamma
    \]
    for each \(T \gg  1\) and \(\theta > 0\) close to zero. The right-hand integral is unbounded as \(\theta \to 0^+\) and consequently \(h \not\in H^2_\omega\).
\end{proof}

\begin{remark}\label{reDir}
Notice that, in \eqref{eqnD0}, the term $h_{n+1}(n+1)$ for $n=-1$ is null. We will come back to this critical point later, when discussing the Dirichlet space, since it is the fact that made the differential equation into a homogeneous one and this is the point where our argument will fail to some extent for other weights. \end{remark}

Let us now discuss the range of values $\alpha \in (0,1)$ for which our weights $\binom{n-\alpha}{n}^{-1}$ define a Dirichlet-type space.

\begin{theorem}\label{thm:dalpha01}
	Let $a \in \mathbb{C}\setminus\{0\}$, $\alpha \in (0,1)$ and $V_{a} = M_{a-z}^*M_{a-z}$ acting on $H^2_\omega$ where $\omega_n = \binom{n-\alpha}{n}^{-1}$. Then:
    \begin{enumerate}[label=\normalfont{(\roman*)}]
        \item If \(0 < |a| \leq 1\),
        \[
		\sigma_{\text{\normalfont p}}(V_{a}) = \Bigg\{\lambda_{j}^+ \coloneqq |a|^2 + 1 + \frac{\alpha^2 + (2j +2 - \alpha)\varrho_{j}}{2(j+1)(j+1-\alpha)} \colon j \in \mathbb{N}\Bigg\},
	\]
	where \(\varrho_{j} \coloneqq \sqrt{\alpha^2 +4|a|^2(j+1)(j+1-\alpha)}\). Eigenspaces are one-dimensional and
    \[
        \ker(V_{a} - \lambda_{j}^+I) = \mathrm{span}\left\{\frac{\Big(z +\mfrac{\varrho_{j} - \alpha}{2\overline{a}(j+1)}\Big)^j}{\Big(1+\mfrac{\varrho_{j} - \alpha}{2\overline{a}(j+1)}z\Big)^{j+2-\alpha}}\right\}.
    \]
        \item If \(|a| > 1\), 
        \[
		\sigma_{\text{\normalfont p}}(V_{a}) = \Bigg\{\lambda_{j}^\pm \coloneqq |a|^2 + 1 + \frac{\alpha^2 \pm (2j +2 - \alpha)\varrho_{j}}{2(j+1)(j+1-\alpha)} : j \in \mathbb{N}\Bigg\}.
	\]
    Moreover, all the eigenspaces are one-dimensional with
    \[
        \ker(V_{a} - \lambda_{j}^\pm I) = \mathrm{span}\left\{\frac{\Big(z\pm\mfrac{\varrho_{j} \mp \alpha}{2\overline{a}(j+1)}\Big)^j}{\Big(1 \pm\mfrac{\varrho_{j} \mp \alpha}{2\overline{a}(j+1)}z\Big)^{j+2-\alpha}}\right\}.
    \]
    \end{enumerate}
\end{theorem}

\begin{proof}
    Assuming \(a > 0\), the proof follows the same steps than for Theorem \ref{thm:dalphanegativo} until the expression \eqref{eq:cambio-demo-dalpha}. Once again, in order to have $|c|>1$ (and then, lie in \textbf{Case 2}), we need \(2a(j+1) > |\alpha \mp \varrho_{j}|\):
    \begin{itemize}
        \item The positive root \(+\varrho_{j}\) for the candidate of eigenvalue in \eqref{eigenvaluescandidates-dalpha} corresponds to
        \[
            c = \frac{2a(j+1)}{\alpha - \varrho_{j}}.
        \]
        Since \(\varrho_j > \alpha\) still holds for \(\alpha \in (0,1)\), one has \(|\alpha - \varrho_j| = \varrho_j - \alpha\). Accordingly, \(|c|>1\) yields the inequality \(2a(j+1)>\varrho_{j} - \alpha\), which is true for all \(a > 0\). Therefore, this choice grants the existence of eigenvalues \(\lambda_j^+\) in the two cases \(0 < |a| < 1\) and \(|a| \geq 1\) detailed in the statement with the corresponding eigenfunctions.
        \item On the other hand, the choice of the negative root \(-\varrho_{j}\) for the eigenvalue in \eqref{eigenvaluescandidates-dalpha} leads to
        \[
            c = \frac{2a(j+1)}{\alpha + \varrho_{j}},
        \]
        whose corresponding inequality is \(2a(j+1)>\varrho_{j} + \alpha\). Assuming \(j > \tfrac{\alpha}{2a} -1\), the solution of the inequality \(2a(j+1) - \alpha>\varrho_{j}\) is the interval \(a > 1\) which, indeed, allows to take \(j \geq 0\). Consequently, this choice yields the existence of eigenvalues \(\lambda_j^-\) only in the latter case \(|a|>1\) considered in the statement with the associated eigenfunctions.
        \end{itemize}\end{proof}

\bigskip

\section{The Dirichlet space case} \label{Dirichlet}

When studying in detail the case of the Dirichlet space $D=H^2_\omega$, where $\omega_n  = n+1$, we found several issues that became hard to solve nicely. These arise from the lack of homogeneity of the differential equation alluded to in Remark \ref{reDir}. As a result, the already complicated situation with 3 different cases regarding the position of the parameter $c$ in Theorem \ref{thm:dalphanegativo} becomes far more cumbersome, relying on too many auxiliary previous results to be able to check the validity of all claims. We therefore state our observations in this area as an open problem and we present a scheme of a potential proof. Our scheme may be correct but we consider it possible for soft arguments to provide better insight. For some of the arguments, we could have focused on $\alpha=1$ only, but the observed difference between the point spectra for different (equivalent) renormings of the same spaces (when $\alpha \in (0,1)$) seems interesting in itself:

\begin{conjecture}\label{thm:d1}
	Let $a \in \mathbb{C}\setminus\{0\}$ and $\alpha > 0$. Consider the operator $V_{a} = M_{a-z}^*M_{a-z}$ acting on $H^2_\omega=D_\alpha$ equipped with the norm corresponding to the weight $\omega_n = \binom{n+\alpha}{n}$. Then $\sigma_{\text{\normalfont{p}}}(V_a) = \emptyset$. \end{conjecture}

The question makes sense as well for $\alpha \in (-1,0)$ but we are not going to pursue this any further.

\begin{proof}[Outline of a potential proof]
Assume $a> 0$. Our strategy is the same than in the proof of Theorem \ref{thm:dalphanegativo}. Following the same steps, we may transform the recurrence relation from Lemma \ref{lemma-recurrencia} into an equation on holomorphic functions:
	\[
	\sum_{n=0}^{\infty} z^{n} (n+1)\cdot (h_{n+1} -  bh_n + h_{n-1}) = -\alpha \sum_{n=0}^{\infty}z^{n} \bigg(h_{n+1}- \frac{h_n}{a}\bigg).
	\]
	The right-hand side above is \(-\alpha\big(\tfrac{h(z) - h(0)}{z} - \tfrac{1}{a}h(z)\big)\), while the left-hand side can be identified with $h' -b (zh)' + (z^2h)'$. Separating terms on $h$ and $h'$, we have a first-order ODE that, after some algebraic manipulations becomes 
	\begin{equation}\label{eq:differential_eq_dirichlet}
		h'(z) + h(z) \frac{\alpha - \big(b + \tfrac{\alpha}{a}\big)z + 2z^2}{zP(z)} = \frac{\alpha h(0)}{zP(z)}.
	\end{equation}
	
Denote by $c$ and $1/c$ the complex roots of $P$ and assume $c \neq 1/c$ (i.e. that $c \neq \pm 1$), and $|c| \geq 1$. Now, in order to solve the equation, we multiply both sides by the integrating factor
	\[
		K(z) = z^\alpha\bigg(1-\frac{z}{c}\bigg)^{\mu}(1-cz)^{\nu},
	\]
	where the exponents $\mu, \nu \in \mathbb{C}$ must be
	\[
		\mu = 1 + \frac{\alpha}{a}\frac{a-c}{c^2-1} = 1 + \frac{\alpha}{a}\frac{a-c}{bc -2}\qquad \text{and}\qquad \nu =1-\alpha -\frac{\alpha}{a} \frac{a-c}{c^2 -1} =  1-\alpha -\frac{\alpha}{a} \frac{a-c}{bc - 2},
	\]
	fulfill the relation $\mu+\nu = 2 - \alpha$ and allow to transform the differential equation into
 \begin{equation}\label{eq:ode-desarrollada}
    h'(z) + h(z)\bigg(\frac{\alpha}{z}  + \frac{\nu}{z -1/c} + \frac{\mu}{z - c}\bigg) = h(0)\left(\frac{\alpha}{z}  - \frac{\tfrac{\alpha c^2}{c^2 -1}}{z -1/c} + \frac{\tfrac{\alpha}{c^2 -1}}{z - c}\right).
 \end{equation}
 Therefore, the general solution to our differential equation \eqref{eq:differential_eq_dirichlet} is
	\[
		h(z) = \frac{\alpha h(0) \medint\int \mfrac{K(z)}{zP(z)}\, \mathrm{d}z + k}{K(z)} = \frac{\alpha h(0) \medint\int z^{\alpha-1}\big(1-\tfrac{z}{c}\big)^{\mu-1}(1-cz)^{\nu-1} \, \mathrm{d}z + k}{z^\alpha\big(1-\tfrac{z}{c}\big)^{\mu}(1-cz)^{\nu}}, \quad \text{with } k \in \mathbb{C}.
	\]
	
	The latter primitive can be expressed in terms of hypergeometric functions. Recall that for each $\beta \in \mathbb{C}$ and $n \geq 0$, the \emph{Pochhammer symbol} (or \emph{rising factorial}) is given by
    \[
        (\beta)_n \coloneqq \begin{cases} \beta(\beta +1 ) \cdots (\beta + n-1), & \text{if } n = 1,2,3,\ldots \\ 1, &\text{otherwise}.\end{cases}
    \]
    The \emph{Gauss hypergeometric function} ${}_2 F_1(\beta,\gamma;\zeta;z)$ is defined on $\mathbb{D}$ by the power series
    \[
        {}_2 F_1(\beta,\gamma;\zeta;z) \coloneqq \sum_{n = 0}^\infty \frac{(\beta)_n (\gamma)_n}{(\zeta)_n} \frac{z^n}{n!}, \qquad |z| < 1
    \]
    for every $\beta, \gamma \in \mathbb{C}$ and $\zeta \in \mathbb{C}\setminus\{0,-1,-2,\ldots\}$. Whenever $\beta$ or $\gamma$ belong to $\{0,-1,-2,\ldots\}$, the function ${}_2 F_1(\beta,\gamma;\zeta;z)$ reduces to a polynomial. If $\beta,\gamma \not\in \{0,-1,-2,\ldots\}$, it can be analytically continued to $\mathbb{C}\setminus [1,+\infty)$. Additionally, we must consider the \emph{first Appell hypergeometric function in two variables}
	\begin{equation}\label{eq:appell}
		F_1(\beta;\gamma,\gamma'; \zeta; x,y) \coloneqq \sum_{m,n = 0}^\infty \frac{(\beta)_{m+n}(\gamma)_m(\gamma')_n}{(\zeta)_{m+n}}\frac{x^m}{m!} \frac{y^n}{n!}, \quad (x,y) \in \mathbb{D}^2
	\end{equation}
	where the parameters $\beta, \gamma, \gamma' \in \mathbb{C}$ and $\zeta \in \mathbb{C}\setminus \{0,-1,-2,\ldots\}$.  We leave the case $\nu =0$ to the reader and suppose $\nu \neq 0$. Then, the primitive
 \[
    \int z^{\alpha-1}\bigg(1-\frac{z}{c}\bigg)^{\mu-1}(1-cz)^{\nu-1} \, \mathrm{d}z
 \]
 is equal to
\[
     \mfrac{z^\alpha}{\alpha} F_1\Big(\alpha; 1 - \mu, 1-\nu; \alpha+1;\mfrac{z}{c},cz\Big)+k, \]  with $k \in \mathbb{C}$. Thus, taking
 \[
    F(z) = F_1\Big(\alpha; 1 - \mu, 1-\nu; \alpha+1;\mfrac{z}{c},cz\Big)\]
the general solution of the differential equation \eqref{eq:differential_eq_dirichlet} is
 \begin{equation}\label{eq:general-solution-dirichlet}
		h(z) =  \frac{h(0)z^\alpha F(z) + k}{z^\alpha\big(1-\tfrac{z}{c}\big)^{\mu}(1-cz)^{\nu}}, \quad \text{with } k \in \mathbb{C}.
\end{equation}

Observe that, since $\alpha>0$, for this function $h(z)$ to be holomorphic at the origin we must take $k = 0$, yielding a candidate eigenfunction (taking $h(0) = 1$):
	\begin{equation}\label{eq:candidate-eigenfunction-dirichlet}
		h(z) =  \frac{F(z)}{\big(1-\tfrac{z}{c}\big)^{\mu}(1-cz)^{\nu}}.
	\end{equation}
	This $h$ is holomorphic, at least, on \(|z| < 1/|c|\). In conclusion, $\lambda \in \sigma_{\text{\normalfont p}}(V_{a}) \colon D_\alpha \to D_\alpha$ if and only if the function $h$ in \eqref{eq:candidate-eigenfunction-dirichlet} belongs to $D_\alpha$. From here, some references that may contribute to a solution include \cite{encyclopedia-special-functions, anal-cont, erdelyi, handbook-hypergeometric, hurwitz, klein, olsson, ReedSimon, runckel, vanbleck}.  Depending on $c$ there are at least 2 cases:

	\medskip
	\begin{itemize}
		\item\textbf{Case 1:} $|c| = 1$.
	\end{itemize}
	
	\smallskip

The singularities at the boundary of the candidate eigenfunction should be too large for $h$ to be in $D_\alpha$.
    
    \medskip
    \begin{itemize}
		\item\textbf{Case 2:} $|c|>1$.
    \end{itemize}

    \smallskip

This is the most intricate case. We found no way around studying subcases in terms of whether $\nu \in \R \backslash \Z$, and of the sign of $\nu$.
\end{proof}

\bigskip

\section{Further remarks} \label{furthersect}

We conclude now by mentioning some directions of potential research.

\begin{remark}
\begin{enumerate}
\item[(a)] Theorem 1 from \cite{killip-simon-annals}, that we used as a reference for Theorem \ref{thmKilSim}, is a stronger statement: it characterizes Hilbert-Schmidt perturbations of free Jacobi matrices as those that satisfy 4 criteria (two of them are the Blumenthal-Weyl criterion and Lieb-Thirring bound). The other two conditions refer to the spectral measure and these are also satisfied by our corresponding measures. Determining the spectral measure seems an interesting future objective and, for that end, one needs to take into account the whole of Killip and Simon's Theorem. In particular, information about the absolutely continuous part of the spectral measure could be contained in the choice of $a$ value determining $V_a$. In \cite{killip-simon-annals}, the authors also provide information about trace class perturbations of the free Jacobi matrix. The Bergman-type weights will not provide trace-class perturbations but rather Schatten class for all $p>1$ (and so will Dirichlet-type weights). Perhaps something stronger can be said about the spectral measure from that information than what follows from their Theorem 1.

\item[(b)] Regarding extension to other functions, the methods that we described here make a strong use of the fact that $V_a$ is given by a polynomial $f$ of degree 1. When increasing the degree of $f$ to $d \in \N$, $d \geq 2$, one can expect some recurrence relation with $2d+1$ terms to describe the properties of eigenfunctions. The decomposition of a polynomial in its irreducible factors could be leveraged to make use of Jacobi matrices properties. If so, then a description of the pheonomenon for all polynomials could be a reasonable aim. If this or any other program for studying the spectra of $M^*_fM_f$ for polynomial $f$ is succesful, then a next reasonable step is to check spectral properties of $V_{f_n}$ that are preserved by convergence of a sequence of different polynomial symbols $\{f_n\}$ to a function $f \in H^2_\omega$. Is there any kind of stability of the properties of the spectrum under strong convergence hypotheses? Cyclic functions are preserved by certain kinds of convergence \cite{AguileraSeco} and their defining features in terms of spectrum should also be preserved.

\item[(c)] The operator $V:=M^*_fM_f$ has been largely studied in the context of Toeplitz operators acting on Hardy spaces. Thus, there is a large class of problems waiting to be studied on other spaces $H$. This is specifically interesting to the authors when $H$ is the Dirichlet space. If we denote $\hat{V}$ the corresponding Toeplitz operator, examples  of what we want to call our attention to include studying properties of $V-\hat{V}$ such as: compactness, Hilbert-Schmidt, trace-class or other $p-$Schatten classes membership. One could also explore the relations between model spaces and kernels of $M^*_f$; or the effects of the positivity of $I-V/\|f\|^2_{M_H}$, where $\|\cdot\|_{M_H}$ is the multiplier norm.

\item[(d)] Approximation theory has been applied in the study of cyclic functions leading to the introduction of so-called optimal polynomial approximants (or \emph{opa}, see \cite{RevMatIber} for instance). These are, given $f$, the polynomials making $p_nf$ the orthogonal projections of $1$ onto a sequence of spaces $H_n$ that exhaust the invariant subspace generated by $f$, space usually denoted by $[f]$. The convergence of these orthogonal projections to $1$ is equivalent to the cyclicity of the function. Coefficients of opa are given by matrices $A_n$ that, when $f$ is a multiplier, happen to represent the action over monomials of the truncations of $M^*_fM_f$. Bearing in mind Theorem 4.1 in \cite{AguileraSeco}, one may be able to obtain additional spectral information about $V_f$ from properties of the matrices $A_n$. 

\item[(e)] Many Hilbert spaces of analytic functions over the unit disc do not satisfy our assumptions here and we do not know what the spectrum of $V_a$ will look like in such cases but some of our observations will extend, at least when the polynomials form an orthogonal basis. Non-extreme de Branges-Rovnyak or other reproducing kernel Hilbert spaces in which cyclicity may be studied for polynomial functions are susceptible to analogue analysis but the situation is expected to be more complicated.

\item[(f)] Some of the techniques here may be used to deal with other multiplier functions $f$, but we do not know what the natural framework will be in order to study our operator $V$ for functions $f$ that are not multipliers, for which $M_f$ will only be densely defined. Then the cyclic behaviour of $f$ in $H^2_\omega$ is still expressed in the information about the action of $V$ on monomials, and since monomials are themselves multipliers, one may find a solution based on classical techniques for densely defined operators, but we encourage the reader to be wary of this potential problem. 
\end{enumerate}
\end{remark}

At this point, the road ahead promises enough mathematical adventures.

\end{document}